\documentclass[]{amsart}

\usepackage[english]{babel}
\usepackage[utf8]{inputenc}
\usepackage{paralist}
\usepackage{amsmath, amsfonts, amssymb, amsthm}
\usepackage{color}
\usepackage{ulem}
\usepackage{tikz}

\newcommand{\CC}{\mathbb{C}}

\newcommand{\NN}{\mathbb{N}}

\newcommand{\ZZ}{\mathbb{Z}}

\newcommand{\Hc}{\mathcal{H}}

\newcommand{\Oc}{\mathcal{O}}

\newcommand{\set}[1]{\left\{ #1 \right\}}
\newcommand{\setb}[1]{\left( #1 \right)}
\newcommand{\abs}[1]{\left| #1 \right|}

\newcommand{\norm}[1]{\left\lVert #1 \right\rVert}
\newcommand{\gfr}{\mathfrak{g}}

\newcommand{\xbf}{\mathbf{x}}
\newcommand{\bino}[2]{\begin{pmatrix} #1 \\ #2 \end{pmatrix}}

\newtheorem{mymasterthm}{notForUse}
\theoremstyle{definition}

\theoremstyle{plain}
\newtheorem{mylemma}[mymasterthm]{Lemma}
\newtheorem{mythm}[mymasterthm]{Theorem}
\newtheorem{mycorol}[mymasterthm]{Corollary}
\newtheorem{myprop}[mymasterthm]{Proposition}

\allowdisplaybreaks

\title[Growth of Linear Recurrences in Function Fields]{On the Growth of Linear Recurrences in Function Fields}

\author[C. Fuchs]{Clemens Fuchs}
\author[S. Heintze]{Sebastian Heintze}
\thanks{Supported by Austrian Science Fund (FWF): I4406.}

\address{University of Salzburg\newline
	\indent Department of Mathematics\newline
	\indent Hellbrunnerstr. 34 \newline
	\indent A-5020 Salzburg, Austria}
\email{clemens.fuchs@sbg.ac.at, sebastian.heintze@sbg.ac.at}

\begin{document}
	
	\maketitle
	
	\begin{abstract}
		Let $ (G_n)_{n=0}^{\infty} $ be a non-degenerate linear recurrence sequence with power sum representation $ G_n = a_1(n) \alpha_1^n + \cdots + a_t(n) \alpha_t^n $. In this paper we will prove a function field analogue of the well known result that in the number field case, under some non-restrictive conditions, for $ n $ large enough the inequality $ \abs{G_n} \geq \left( \max_{j=1,\ldots,t} \abs{\alpha_j} \right)^{n(1-\varepsilon)} $ holds true.
	\end{abstract}
	
	\section{Introduction}
	
	Let $ (G_n)_{n=0}^{\infty} $ be a non-degenerate linear recurrence sequence with power sum representation $ G_n = a_1(n) \alpha_1^n + \cdots + a_t(n) \alpha_t^n $.
	This expression makes sense for a sequence $ (G_n)_{n=0}^{\infty} $ taking values in any field $ K $; the characteristic roots $ \alpha_i $ as well as the coefficients of the polynomials $ a_i $ then lie in a finite extension $ L $ of $ K $. In this paper $ K $ is either a number field or a function field in one variable of characteristic zero.
	The condition non-degenerate means in the number field case that no ratio $ \alpha_i/\alpha_j $ for $ i \neq j $ is a root of unity and in the function field case that no ratio $ \alpha_i/\alpha_j $ for $ i \neq j $ is contained in the field of constants.
	In the number field case it is well known that, if $ \max_{j=1,\ldots,t} \abs{\alpha_j} > 1 $, then for any $ \varepsilon > 0 $ the inequality
	\begin{equation}
		\label{p3-eq:ineqofinterest}
		\abs{G_n} \geq \left( \max_{j=1,\ldots,t} \abs{\alpha_j} \right)^{n(1-\varepsilon)}
	\end{equation}
	is satisfied for every $ n $ sufficiently large.
	
	The purpose of this paper is to prove an analogous result in the case of a function field in one variable of characteristic zero. Firstly, we will prove a theorem which states an inequality for an arbitrary valuation in the splitting field $ L $ of the characteristic polynomial belonging to the linear recurrence sequence.
	Secondly, we will derive a corollary for the special case of polynomial power sums. In this special case the inequality takes a form very similar to \eqref{p3-eq:ineqofinterest}.
	At this point we remark the following: Theorem 1 in \cite{fuchs-petho-2005} already implies that under some non-degeneracy conditions the degree of polynomials in a linear recurrence sequence with polynomial roots cannot be bounded and therefore must grow up to infinity as $ n $ does. But this theorem does not say how fast it must grow. In the present paper we will give a bound depending on $ n $ for the minimal possible degree of $ G_n $.
	
	Since the number field case is often mentioned, e.g. in \cite{vanderpoorten-1984} or in \cite{evertse-1984} or more recently in \cite{akiyama-evertse-petho-2017} and \cite{bugeaud-kaneko-2019}, but it is not that easy to access a proof of it (cf. \cite{vanderpoorten-schlickewei-1982} or the formulation in \cite{vanderpoorten-1989}), we give a complete proof based on results of Evertse and Schmidt in the appendix.
	By doing so we contribute to the goal that well known facts should be fully accessible with proof following van der Poorten's wise statement in \cite{vanderpoorten-1989} that ``all too frequently, [that] the well known is [often] not generally known, let alone known well''.
	
	\section{Results and notations}
	
	Throughout the paper we denote by $ K $ a function field in one variable over $ \CC $. By $ L $ we usually denote a finite algebraic extension of $ K $.
	For the convenience of the reader we give a short wrap-up of the notion of valuations that can e.g. also be found in \cite{fuchs-heintze-p2}:
	For $ c \in \CC $ and $ f(x) \in \CC(x) $, where $ \CC(x) $ is the rational function field over $ \CC $, we denote by $ \nu_c(f) $ the unique integer such that $ f(x) = (x-c)^{\nu_c(f)} p(x) / q(x) $ with $ p(x),q(x) \in \CC[x] $ such that $ p(c)q(c) \neq 0 $. Further we write $ \nu_{\infty}(f) = \deg q - \deg p $ if $ f(x) = p(x) / q(x) $.
	These functions $ \nu : \CC(x) \rightarrow \ZZ $ are up to equivalence all valuations in $ \CC(x) $.
	If $ \nu_c(f) > 0 $, then $ c $ is called a zero of $ f $, and if $ \nu_c(f) < 0 $, then $ c $ is called a pole of $ f $, where $ c \in \CC \cup \set{\infty} $.
	For a finite extension $ K $ of $ \CC(x) $ each valuation in $ \CC(x) $ can be extended to no more than $ [K : \CC(x)] $ valuations in $ K $. This again gives up to equivalence all valuations in $ K $.
	Both, in $ \CC(x) $ as well as in $ K $ the sum-formula
	\begin{equation*}
		\sum_{\nu} \nu(f) = 0
	\end{equation*}
	holds, where the sum is taken over all valuations in the considered function field.
	Moreover, valuations have the properties $ \nu(fg) = \nu(f) + \nu(g) $ and $ \nu(f+g) \geq \min \set{\nu(f), \nu(g)} $ for all $ f,g \in K $.
	Each valuation in a function field corresponds to a place and vice versa.
	The places can be thought as the equivalence classes of valuations.
	For more information about valuations and places we refer to \cite{stichtenoth-1993}.
	
	For any power sum $ G_n = a_1(n) \alpha_1^n + \cdots + a_t(n) \alpha_t^n $ with $ a_j(n) = \sum_{k=0}^{m_j} a_{jk} n^k $ and any valuation $ \mu $ (in a function field $ L/K $ containing the $ \alpha_j $ and the coefficients of the $ a_j $) we have the trivial bound
	\begin{align*}
		\mu(G_n) &= \mu\left( a_1(n) \alpha_1^n + \cdots + a_t(n) \alpha_t^n \right)
		\geq \min_{j=1,\ldots,t} \mu(a_j(n) \alpha_j^n) \\
		&\geq \min_{j=1,\ldots,t} \mu(a_j(n)) + \min_{j=1,\ldots,t} \mu(\alpha_j^n) \\
		&\geq \min_{j=1,\ldots,t} \min_{k=0,\ldots,m_j} \mu(a_{jk} n^k) + n \cdot \min_{j=1,\ldots,t} \mu(\alpha_j) \\
		&= \min_{\substack{j=1,\ldots,t \\ k=0,\ldots,m_j}} \mu(a_{jk}) + n \cdot \min_{j=1,\ldots,t} \mu(\alpha_j)
		= \widetilde{C} + n \cdot \min_{j=1,\ldots,t} \mu(\alpha_j).
	\end{align*}
	Our main result is now the following theorem which gives a bound in the other direction:
	
	\begin{mythm}
		\label{p3-thm:funcfieldcase}
		Let $ (G_n)_{n=0}^{\infty} $ be a non-degenerate linear recurrence sequence taking values in $ K $ with power sum representation $ G_n = a_1(n) \alpha_1^n + \cdots + a_t(n) \alpha_t^n $. Let $ L $ be the splitting field of the characteristic polynomial of that sequence, i.e. $ L = K(\alpha_1, \ldots, \alpha_t) $. Moreover, let $ \mu $ be a valuation on $ L $. Then there is an effectively computable constant $ C $, independent of $ n $, such that for every sufficiently large $ n $ the inequality
		\begin{equation*}
			\mu(G_n) \leq C + n \cdot \min_{j=1,\ldots,t} \mu(\alpha_j)
		\end{equation*}
		holds.
	\end{mythm}
	
	For the special case of a linear recurrence sequence of complex polynomials having complex polynomials as characteristic roots we get the following lower bound for the degree of the $ n $-th member of the sequence:
	
	\begin{mycorol}
		\label{p3-corol:polypowersum}
		Let $ (G_n)_{n=0}^{\infty} $ be a non-degenerate linear recurrence sequence of polynomials in $ \CC[x] $ with power sum representation $ G_n = a_1(n) \alpha_1^n + \cdots + a_t(n) \alpha_t^n $ such that $ \alpha_1, \ldots, \alpha_t \in \CC[x] $. Then there is an effectively computable constant $ C $, independent of $ n $, such that for every sufficiently large $ n $ the inequality
		\begin{equation*}
			\deg G_n \geq n \cdot \max_{j=1,\ldots,t} \deg \alpha_j - C
		\end{equation*}
		holds.
	\end{mycorol}
	
	It would be interesting to also prove a function field variant of Corollary 3.1 in \cite{akiyama-evertse-petho-2017}. However, because of Lemma 3.1 therein, which is based on Dirichlet's classical approximation theorem, we are not (yet) able to prove such a statement.
	
	In the proof given in the next section we are going to apply the subsequently given special case of Theorem 1 in \cite{fuchs-petho-2005}:
	\begin{myprop}
		\label{p3-prop:luorbounded}
		Let $ K $ be as above and $ L $ be a finite extension of $ K $ of genus $ \gfr $. Let further $ \alpha_1, \ldots, \alpha_d \in L^* $ with $ d \geq 2 $ be such that $ \alpha_i / \alpha_j \notin \CC^* $ for each pair of subscripts $ i,j $ with $ 1 \leq i < j \leq d $. Moreover, for every $ i = 1, \ldots, d $ let $ \pi_{i1}, \ldots, \pi_{ir_i} \in L $ be $ r_i $ linearly independent elements over $ \CC $. Put
		\begin{equation*}
			q = \sum_{i=1}^{d} r_i.
		\end{equation*}
		Then for every $ n \in \NN $ such that
		\begin{equation*}
			\set{ \pi_{il} \alpha_i^n : l=1,\ldots,r_i, i=1,\ldots,d }
		\end{equation*}
		is linearly dependent over $ \CC $, but no proper subset of this set is linearly dependent over $ \CC $, we have
		\begin{equation*}
			n \leq C = C(q, \gfr, \pi_{il}, \alpha_i : l=1,\ldots,r_i, i=1,\ldots,d).
		\end{equation*}
	\end{myprop}
	
	The proof will also make use of height functions in function fields. Let us therefore define the height of an element $ f \in L^* $ by
	\begin{equation*}
		\Hc(f) := - \sum_{\nu} \min \setb{0, \nu(f)} = \sum_{\nu} \max \setb{0, \nu(f)}
	\end{equation*}
	where the sum is taken over all valuations in the function field $ L / \CC $.
	Observe that for every $ z \in L \setminus \CC $ we have $ \Hc(z) = \sum_{\nu} \max \setb{0, \nu(z)} = \sum_{P} \max \setb{0, \nu_P(z)} = \deg \sum_{P} \max \setb{0, \nu_P(z)} P = \deg (z)_0 = [L:\CC(z)] = \deg_{\CC}(z) $, by Theorem I.4.11 in \cite{stichtenoth-1993}, where we have used that all places have degree one since we are working over $ \CC $; therefore instead of the height one can use $ \deg_{\CC}(z) = [L:\CC(z)]$ (as in \cite{zannier-2008}).
	Additionally we define $ \Hc(0) = \infty $.
	This height function satisfies some basic properties that are listed in the lemma below which is proven in \cite{fuchs-karolus-kreso-2019}:
	
	\begin{mylemma}
		\label{p3-lemma:heightproperties}
		Denote as above by $ \Hc $ the height on $ L/\CC $. Then for $ f,g \in L^* $ the following properties hold:
		\begin{enumerate}[a)]
			\item $ \Hc(f) \geq 0 $ and $ \Hc(f) = \Hc(1/f) $,
			\item $ \Hc(f) - \Hc(g) \leq \Hc(f+g) \leq \Hc(f) + \Hc(g) $,
			\item $ \Hc(f) - \Hc(g) \leq \Hc(fg) \leq \Hc(f) + \Hc(g) $,
			\item $ \Hc(f^n) = \abs{n} \cdot \Hc(f) $,
			\item $ \Hc(f) = 0 \iff f \in \CC^* $,
			\item $ \Hc(A(f)) = \deg A \cdot \Hc(f) $ for any $ A \in \CC[T] \setminus \set{0} $.
		\end{enumerate}
	\end{mylemma}
	
	Furthermore we will use the following function field analogue of the Schmidt subspace theorem. A proof can be found in \cite{zannier-2008}:
	\begin{myprop}[Zannier]
		\label{p3-prop:functionfieldsubspace}
		Let $ F/\CC $ be a function field in one variable, of genus $ \gfr $, let $ \varphi_1, \ldots, \varphi_n \in F $ be linearly independent over $ \CC $ and let $ r \in \set{0,1, \ldots, n} $. Let $ S $ be a finite set of places of $ F $ containing all the poles of $ \varphi_1, \ldots, \varphi_n $ and all the zeros of $ \varphi_1, \ldots, \varphi_r $. Put $ \sigma = \sum_{i=1}^{n} \varphi_i $. Then
		\begin{equation*}
			\sum_{\nu \in S} \left( \nu(\sigma) - \min_{i=1, \ldots, n} \nu(\varphi_i) \right) \leq \bino{n}{2} (\abs{S} + 2\gfr - 2) + \sum_{i=r+1}^{n}  \Hc(\varphi_i).
		\end{equation*}
	\end{myprop}
	
	\section{Proofs}
	
	\begin{proof}[Proof of Theorem \ref{p3-thm:funcfieldcase}]
		Denote the coefficients of the polynomial $ a_j(n) \in L[n] $ by $ a_{j0}, a_{j1}, \ldots, a_{jm_j} $ where $ m_j $ is the degree of $ a_j(n) $. So
		\begin{equation*}
			a_j(n) = \sum_{k=0}^{m_j} a_{jk} n^k.
		\end{equation*}
		
		First assume that the recurrence sequence is of the shape $ G_n = a_1(n) \alpha_1^n $. Using Lemma \ref{p3-lemma:heightproperties} we get
		\begin{align*}
			\mu(G_n) &= \mu(a_1(n)) + n\mu(\alpha_1) \leq \Hc(a_1(n)) + n\mu(\alpha_1) \\
			&\leq \sum_{k=0}^{m_1} \Hc(a_{1k}n^k) + n\mu(\alpha_1) = \sum_{k=0}^{m_1} \Hc(a_{1k}) + n\mu(\alpha_1).
		\end{align*}
		
		Thus from now on we can assume that $ t \geq 2 $.
		Let $ \pi_{j1}, \ldots, \pi_{jk_j} $ be a maximal $ \CC $-linear independent subset of $ a_{j0}, a_{j1}, \ldots, a_{jm_j} $. Then we can write the sequence in the form
		\begin{equation*}
			G_n = \sum_{j=1}^{t} \left( \sum_{i=1}^{k_j} b_{ji}(n) \pi_{ji} \right) \alpha_j^n
		\end{equation*}
		with polynomials $ b_{ji}(n) \in \CC[n] $.
		Since $ a_j(n) $ is not the zero polynomial, there is for each $ j $ at least one index $ i $ such that $ b_{ji}(n) $ is not the zero polynomial.
		Without loss of generality we can assume that no $ b_{ji}(n) $ is the zero polynomial since otherwise we can throw out all zero polynomials and renumber the remaining terms. It does not matter whether all $ \pi_{ji} $ occur in the sum or not.
		Moreover we assume that $ n $ is large enough such that $ b_{ji}(n) \neq 0 $ for all $ j,i $.
		
		Consider as a next step the set
		\begin{equation*}
			M := \set{\pi_{ji} \alpha_j^n : i=1,\ldots,k_j, j=1,\ldots,t}.
		\end{equation*}
		We intend to apply Proposition \ref{p3-prop:luorbounded}.
		If $ M $ is linearly dependent over $ \CC $, then we choose a minimal linear dependent subset $ \widetilde{M} $ of $ M $, i.e. a linearly dependent subset $ \widetilde{M} $ with the property that no proper subset of $ \widetilde{M} $ is linearly dependent. Let $ \widetilde{G_n} $ be the linear recurrence sequence associated with this subset $ \widetilde{M} $, that is
		\begin{equation*}
			\widetilde{G_n} = \sum_{j=1}^{s} \left( \sum_{i=1}^{\widetilde{k_j}} b_{ji}(n) \pi_{ji} \right) \alpha_j^n
		\end{equation*}
		for $ s \leq t $ and after a suitable renumbering of the summands. Since $ \pi_{j1}, \ldots, \pi_{jk_j} $ are $ \CC $-linear independent we have $ s \geq 2 $.
		Applying Proposition \ref{p3-prop:luorbounded} to
		\begin{equation*}
			\widetilde{M} := \set{\pi_{ji} \alpha_j^n : i=1,\ldots,\widetilde{k_j}, j=1,\ldots,s}
		\end{equation*}
		gives an upper bound for $ n $. Thus for $ n $ large enough this subset $ \widetilde{M} $ of $ M $ cannot be linearly dependent.
		Because of the fact that there are only finitely many subsets of $ M $, therefore for $ n $ large enough the set $ M $ must be linearly independent.
		
		We assume from here on that $ n $ is large enough such that $ M $ is linearly independent.
		For each fixed $ n $ we have $ b_{ji}(n) \in \CC^* $. Thus the set
		\begin{equation*}
			M' := \set{b_{ji}(n) \pi_{ji} \alpha_j^n : i=1,\ldots,k_j, j=1,\ldots,t}.
		\end{equation*}
		is linearly independent over $ \CC $ and contains for each $ j=1,\ldots,t $ at least one element.
		Let $ S $ be a finite set of places of $ L $ containing all zeros and poles of $ \alpha_j $ for $ j=1,\ldots,t $ and of the nonzero $ a_{ji} $ for $ j=1,\ldots,t $ and $ i=1,\ldots,m_j $ as well as $ \mu $ and the places lying over $ \infty $.
		Now applying Proposition \ref{p3-prop:functionfieldsubspace} yields
		\begin{equation*}
			\sum_{\nu \in S} \left( \nu(G_n) - \min_{\substack{j=1,\ldots,t \\ i=1,\ldots,k_j}} \nu\left( b_{ji}(n) \pi_{ji} \alpha_j^n \right) \right) \leq \bino{\sum_{j=1}^{t} k_j}{2} (\abs{S} + 2\gfr - 2) =: C_1
		\end{equation*}
		and since each summand in the sum on the left hand side is non-negative
		\begin{equation*}
			\mu (G_n) - \min_{\substack{j=1,\ldots,t \\ i=1,\ldots,k_j}} \mu\left( b_{ji}(n) \pi_{ji} \alpha_j^n \right) \leq C_1.
		\end{equation*}
		Therefore for all $ j_0=1,\ldots,t $ and $ i_0=1,\ldots,k_{j_0} $ we get
		\begin{align*}
			\mu (G_n) &\leq C_1 + \min_{\substack{j=1,\ldots,t \\ i=1,\ldots,k_j}} \mu\left( b_{ji}(n) \pi_{ji} \alpha_j^n \right) \\
			&\leq C_1 + \mu\left( b_{j_0i_0}(n) \pi_{j_0i_0} \alpha_{j_0}^n \right) \\
			&= C_1 + \mu\left( \pi_{j_0i_0} \right) + n \mu\left( \alpha_{j_0} \right) \\
			&\leq C_1 + \max_{\substack{j=1,\ldots,t \\ i=0,\ldots,m_j ,\ a_{ji} \neq 0}} \mu(a_{ji}) + n \mu\left( \alpha_{j_0} \right) \\
			&\leq C_1 + \max_{\substack{j=1,\ldots,t \\ i=0,\ldots,m_j ,\ a_{ji} \neq 0}} \Hc(a_{ji}) + n \mu\left( \alpha_{j_0} \right) \\
			&= C_2 + n \mu\left( \alpha_{j_0} \right).
		\end{align*}
		Since this holds for all $ j_0=1,\ldots,t $ we have
		\begin{equation*}
			\mu (G_n) \leq C_2 + n \cdot \min_{j=1,\ldots,t} \mu(\alpha_j).
		\end{equation*}
	\end{proof}
	
	\begin{proof}[Proof of Corollary \ref{p3-corol:polypowersum}]
		We can apply Theorem \ref{p3-thm:funcfieldcase} with $ L = K = \CC(x) $ and $ \mu = \nu_{\infty} $. This yields
		\begin{align*}
			-\deg G_n &= \nu_{\infty} (G_n) \leq C + n \cdot \min_{j=1,\ldots,t} \nu_{\infty} (\alpha_j) \\
			&= C - n \cdot \max_{j=1,\ldots,t} \deg \alpha_j
		\end{align*}
		which immediately implies the inequality in question.
	\end{proof}
	
	
	\section*{Appendix. The number field case}
	
	In this appendix we will give a proof for the following theorem:
	\begin{mythm}
		\label{p3-thm:numfieldcase}
		Let $ (G_n)_{n=0}^{\infty} $ be a non-degenerate linear recurrence sequence taking values in a number field $ K $ and let $ G_n = a_1(n) \alpha_1^n + \cdots + a_t(n) \alpha_t^n $ with algebraic integers $ \alpha_1,\ldots,\alpha_t $ be its power sum representation satisfying $ \max_{j=1,\ldots,t} \abs{\alpha_j} > 1 $. Denote with $ \abs{\cdot} $ the usual absolute value on $ \CC $. Then for any $ \varepsilon > 0 $ the inequality
		\begin{equation*}
			\abs{G_n} \geq \left( \max_{j=1,\ldots,t} \abs{\alpha_j} \right)^{n(1-\varepsilon)}
		\end{equation*}
		is satisfied for every $ n $ sufficiently large.
	\end{mythm}
	
	From here on $ K $ will denote a number field.
	In the proof we will need three auxiliary results which are listed below. The first one is a result of Schmidt and can be found in \cite{schmidt-1999}:
	\begin{mythm}
		\label{p3-thm:zerobound}
		Suppose that $ (G_n)_{n \in \ZZ} $ is a non-degenerate linear recurrence sequence of complex numbers, whose characteristic polynomial has $ k $ distinct roots of multiplicity $ \leq a $. Then the number of solutions $ n \in \ZZ $ of the equation
		\begin{equation*}
			G_n = 0
		\end{equation*}
		can be bounded above by
		\begin{equation*}
			c(k,a) = e^{(7k^a)^{8k^a}}.
		\end{equation*}
	\end{mythm}
	The second one is a result of Evertse. The reader will find it as Theorem 2 in \cite{evertse-1984}. We use the notation
	\begin{equation*}
		\norm{\xbf} = \max_{\substack{k=0,\ldots,t \\ i=1,\ldots,D}} \abs{\sigma_i(x_k)}
	\end{equation*}
	with $ \set{\sigma_1, \ldots, \sigma_D} $ the set of all embedings of $ K $ in $ \CC $ and $ \xbf = (x_0,x_1,\ldots,x_t) $.
	Moreover, we denote by $ \Oc_K $ the ring of integers in $ K $:
	\begin{mythm}
		\label{p3-thm:valuationprodineq}
		Let $ t $ be a non-negative integer and $ S $ a finite set of places in $ K $, containing all infinite places. Then for every $ \varepsilon > 0 $ a constant $ C $ exists, depending only on $ \varepsilon, S, K, t $ such that for each non-empty subset $ T $ of $ S $ and every vector $ \xbf = (x_0,x_1,\ldots,x_t) \in \Oc_K^{t+1} $ with
		\begin{equation*}
			x_{i_0} + x_{i_1} + \cdots + x_{i_s} \neq 0
		\end{equation*}
		for each non-empty subset $ \set{i_0,i_1,\ldots,i_s} $ of $ \set{0,1,\ldots,t} $ the inequality
		\begin{equation*}
			\left( \prod_{k=0}^{t} \prod_{\nu \in S} \norm{x_k}_{\nu} \right) \prod_{\nu \in T} \norm{x_0 + x_1 + \cdots + x_t}_{\nu} \geq C \left( \prod_{\nu \in T} \max_{k=0,\ldots,t} \norm{x_k}_{\nu} \right) \norm{\xbf}^{-\varepsilon}.
		\end{equation*}
		is valid.
	\end{mythm}
	Furthermore, we will need the following lemma which also can be found in \cite{evertse-1984}:
	\begin{mylemma}
		\label{p3-lemma:valprodbound}
		Suppose $ K $ is a number field of degree $ D $, let $ f(X) \in K[X] $ be a polynomial of degree $ m $ and $ T $ a non-empty set of primes on $ K $. Then there exists a positive constant $ c $, depending only on $ K,f $ such that for all $ r \in \ZZ $ with $ r \neq 0 $ and $ f(r) \neq 0 $ it holds that
		\begin{align*}
			c^{-1} \abs{r}^{-Dm} &\leq \left( \prod_{\nu} \max \setb{1, \norm{f(r)}_{\nu}} \right)^{-1} \leq \prod_{\nu \in T} \norm{f(r)}_{\nu} \\
			&\leq \prod_{\nu} \max \setb{1, \norm{f(r)}_{\nu}} \leq c \abs{r}^{Dm}.
		\end{align*}
	\end{mylemma}
	
	\begin{proof}[Proof of Theorem \ref{p3-thm:numfieldcase}]
		Since the characteristic roots $ \alpha_j $ of $ G_n $ are algebraic integers we can find a nonzero integer $ z $ such that $ z a_j(n) \alpha_j^n $ are algebraic integers for all $ j=1,\ldots,t $ and all $ n \in \NN $.
		Denote by $ L $ the splitting field of the characteristic polynomial of the sequence $ G_n $, i.e. $ L = K(\alpha_1,\ldots,\alpha_t) $.
		Choose $ S $ as a finite set of places in $ L $ containing all infinite places as well as all places such that $ \alpha_1, \ldots, \alpha_t $ are $ S $-units.
		Let $ \mu $ be such that $ \norm{\cdot}_{\mu} = \abs{\cdot} $ is the usual absolute value on $ \CC $. In particular we have $ \mu \in S $. Further define $ T = \set{\mu} $.
		
		As $ G_n $ is non-degenerate the sequence $ \widetilde{G_n} = zG_n $ is non-degenerate, too. Therefore applying Theorem \ref{p3-thm:zerobound} yields for $ n $ large enough
		\begin{equation*}
			z a_{j_1}(n) \alpha_{j_1}^n + \cdots + z a_{j_s}(n) \alpha_{j_s}^n \neq 0
		\end{equation*}
		for each non-empty subset $ \set{j_1,\ldots,j_s} $ of $ \set{1,\ldots,t} $.
		Thus we can apply Theorem \ref{p3-thm:valuationprodineq} and get
		\begin{equation*}
			\left( \prod_{j=1}^{t} \prod_{\nu \in S} \norm{z a_j(n) \alpha_j^n}_{\nu} \right) \abs{zG_n} \geq C \max_{j=1,\ldots,t} \abs{z a_j(n) \alpha_j^n} \norm{z\xbf}^{-\varepsilon}
		\end{equation*}
		for $ \xbf = (a_1(n) \alpha_1^n, \ldots, a_t(n) \alpha_t^n) $.
		Assuming without loss of generality that $ \abs{\alpha_1} = \max_{j=1,\ldots,t} \abs{\alpha_j} $ and using that $ z $ is a fixed integer as well as that the $ \alpha_j $ are $ S $-units, we can rewrite this as
		\begin{align}
			\left( \prod_{j=1}^{t} \prod_{\nu \in S} \norm{a_j(n)}_{\nu} \right) \abs{G_n} &\geq C_1 \max_{j=1,\ldots,t} \abs{a_j(n) \alpha_j^n} \norm{\xbf}^{-\varepsilon} \nonumber \\
			&\label{p3-eq:appproof1}\geq C_1 \abs{a_1(n) \alpha_1^n} \norm{\xbf}^{-\varepsilon} \\
			&= C_1 \abs{a_1(n)} \abs{\alpha_1}^n \norm{\xbf}^{-\varepsilon}. \nonumber
		\end{align}
		
		In preparation for the next step note that there exists a positive constant $ A $ such that
		\begin{equation*}
			\max_{\substack{j=1,\ldots,t \\ i=1,\ldots,D}} \abs{\sigma_i(\alpha_j)} \leq A \cdot \abs{\alpha_1}.
		\end{equation*}
		We decompose $ \varepsilon = \gamma \cdot \delta $ with small $ \delta $ and $ A^{\gamma} \leq \abs{\alpha_1} $.
		Using this we get the estimates
		\begin{align*}
			\norm{\xbf} &= \max_{\substack{j=1,\ldots,t \\ i=1,\ldots,D}} \abs{\sigma_i(a_j(n) \alpha_j^n)} = \max_{\substack{j=1,\ldots,t \\ i=1,\ldots,D}} \abs{\sigma_i(a_j(n)) \sigma_i(\alpha_j)^n} \\
			&\leq \max_{\substack{j=1,\ldots,t \\ i=1,\ldots,D}} \abs{\sigma_i(a_j(n))} \cdot \max_{\substack{j=1,\ldots,t \\ i=1,\ldots,D}} \abs{\sigma_i(\alpha_j)}^n \\
			&\leq C_2 n^m \cdot \max_{\substack{j=1,\ldots,t \\ i=1,\ldots,D}} \abs{\sigma_i(\alpha_j)}^n \leq C_2 n^m A^n \abs{\alpha_1}^n
		\end{align*}
		with $ m = \max_{j=1,\ldots,t} \deg a_j $, and
		\begin{equation*}
			\norm{\xbf}^{\varepsilon} \leq C_3 n^{m\varepsilon} A^{\gamma n\delta} \abs{\alpha_1}^{n\varepsilon} \leq C_3 n^{m\varepsilon} \abs{\alpha_1}^{n(\varepsilon+\delta)}.
		\end{equation*}
		Now we insert this into inequality \eqref{p3-eq:appproof1} which gives us
		\begin{align*}
			\left( \prod_{j=1}^{t} \prod_{\nu \in S} \norm{a_j(n)}_{\nu} \right) \abs{G_n} &\geq C_4 \abs{a_1(n)} \abs{\alpha_1}^n n^{-m\varepsilon} \abs{\alpha_1}^{-n(\varepsilon+\delta)} \\
			&\geq C_5 n^{-m\varepsilon} \abs{\alpha_1}^{n(1-\varepsilon-\delta)}.
		\end{align*}
		Applying Lemma \ref{p3-lemma:valprodbound} to the product in the brackets on the left hand side gives the bound
		\begin{equation*}
			\prod_{j=1}^{t} \prod_{\nu \in S} \norm{a_j(n)}_{\nu} \leq \prod_{j=1}^{t} C_6^{(j)} n^{Dm} \leq C_7 n^{tDm}.
		\end{equation*}
		Altogether for $ n $ large enough the inequality
		\begin{equation*}
			\abs{G_n} \geq C_8 n^{-tDm-m\varepsilon} \abs{\alpha_1}^{n(1-\varepsilon-\delta)}
		\end{equation*}
		holds.
		Hence, recalling $ \abs{\alpha_1} = \max_{j=1,\ldots,t} \abs{\alpha_j} $, for $ n $ large enough we get
		\begin{equation*}
			\abs{G_n} \geq \left( \max_{j=1,\ldots,t} \abs{\alpha_j} \right)^{n(1-\widetilde{\varepsilon})}.
		\end{equation*}
		This proves the theorem.
	\end{proof}

\end{document}